\documentclass[12pt]{article}
\usepackage{mathrsfs}
\usepackage{epic,eepic,epsf,epsfig}
\usepackage{amsfonts,srcltx,mathrsfs}
\usepackage{multirow}
\usepackage{amsmath}
\usepackage{amssymb}
\usepackage{amsbsy}
\usepackage{graphicx}
\usepackage{amsfonts}
\usepackage{color}
\usepackage{setspace}

\newtheorem{lem}{Lemma}[section]%
\newtheorem{theorem}[lem]{Theorem}%

\def\nd{\mathrel{\bigm|\kern-.7em/}}

\def\f{\noindent}

\def\P\GammaL{\hbox{\rm P\GammaL}}

\def\mod{\hbox{\rm mod }}

\begin{document}
\title{Spectral radius and maximum matching covered
graphs with bounded matching number}

\footnotetext{* Corresponding author}
\footnotetext{E-mails: 13021531326@163.com; zhangwq@pku.edu.cn}

\author{Xinying Tang, Wenqian Zhang*\\
{\small School of Mathematics and Statistics, Shandong University of Technology}\\
{\small Zibo, Shandong 255000, P.R. China}}
\date{}
\maketitle

\begin{abstract}
Let $G$ be a graph. The {\em spectral radius} of $G$ is the largest eigenvalue of its {\em adjacency matrix}. A {\em matching} of $G$ is a set of disjoint edges of $G$. The {\em matching number} of $G$ is the size of a maximum matching (i.e., a matching with  maximum edges). The graph $G$ is called {\em maximum  matching covered} if each edge of $G$ is contained in a maximum matching. In this paper, we give a sharp spectral radius condition for graphs with bounded matching number to be maximum matching covered.
\bigskip

\f {\bf Keywords:} eigenvalue; spectral radius; matching number; maximum matching covered graph.

\end{abstract}

 \baselineskip 17 pt

\section{Introduction}

All graphs considered in this paper are finite, undirected and simple. We first introduce some elementary symbols.
Let $G$ be a graph. The vertex set and edge set of $G$ are denoted by $V(G)$ and $E(G)$, respectively. Set $e(G)=|E(G)|$. For a vertex $u$, let $N_{G}(u)$ denote the set of neighbors of $u$ in $G$, and let $d_{G}(u)=|N_{G}(u)|$. For a $S\subseteq V(G)$, let $G[S]$ denote the subgraph induced by $B$, and  let $G-S=G[V(G)-S]$. For any two disjoint subsets $U,W$ of $V(G)$, let $e_{G}(U,W)$ denote the number of edges between $U$ and $W$ in $G$. For any two graphs $G_{1}$ and $G_{2}$, let $G_{1}\cup G_{2}$ be the {\em disjoint union} of them, and let $G_{1}\vee G_{2}$ be the {\em join} of them (i.e., the graph obtained from $G_{1}$ and $G_{2}$ by connecting each vertex of $G_{1}$ to each vertex of $G_{2}$). For an integer $t\geq1$, let $tG_{1}$ denote the disjoint union of $t$ copies of $G_{1}$. Let $K_{t}$ denote the complete graph of order $t$. For any terminology used but not defined here, one may refer to \cite{BH1,CRS}.

Let $G$ be a graph with vertices $u_{1},u_{2},\ldots,u_{n}$. The {\em adjacency matrix} $A(G)=(a_{ij})_{n\times n}$ of $G$ is an $n\times n$ square matrix, where $a_{ij}=1$ if $u_{i}$ is adjacent to $u_{j}$, and $a_{ij}=0$ otherwise. The eigenvalues of $G$ are the eigenvalues of its adjacency matrix $A(G)$. The {\em spectral radius} $\rho(G)$ of $G$ is the largest eigenvalue of $G$. By Perron-Frobenius Theorem (see \cite{CRS}), $\rho(G)$ has  non-negative eigenvectors (called {\em Perron vectors}).  If $G$ is connected, then $\rho(G)$ is simple and its Perron vector is positive.

Let $G$ be  a graph of order $n$. A {\em matching} $M$ is a set of disjoint edges of $G$. $M$ is called {\em perfect} if $M$ has $\frac{n}{2}$ edges, and is called {\em near perfect} if $M$ has $\frac{n-1}{2}$ edges. A {\em maximum matching} of $G$ is a matching with the maximum edges. The {\em matching number} of $G$ is the size of a maximum matching. The graph
$G$ is called {\em maximum matching covered}, if each edge of $G$ is contained in a maximum matching.

A connected graph $G$ is called {\em factor-critical}, if $G-u$ contains perfect matchings for any $u\in V(G)$. $G$ is called {\em 1-extendable} (or {\em matching covered}), if $G$ has perfect
matchings and each edge of $G$ is contained in a perfect matching. Clearly, a maximum matching covered graph  is also 1-extendable if it has perfect matchings.  Mkrchyan \cite{M}
proved that minimal maximum matching covered graphs without isolated vertices contain 
perfect matchings. The study on matching extensions can be traced back to 1964
when Hetyei \cite{H} studied bipartite 1-extendable graphs. A
necessary and sufficient condition for graphs to be 1-extendable was given by  Little, Grant and Holton \cite{LGH}. It seems to be a hot topic to study 
the structure of 1-extendable graphs (see \cite{CLM,CLM1,CFLLZ,LCKM}). 

It seems to be popular for the study of matchings of graphs in terms of spectral radius (see \cite{FYZ,FL,KOSS,LLNW,MLW,MDW,NLW,O,ZSZ,Z}). In particular, Miao, Li and Wei \cite{MLW} gave a spectral radius condition for graphs to be 1-extendable. Very recently, Niu, Lian and Wang \cite{NLW} gave a spectral radius condition for graphs without perfect matchings to be maximum matching covered. Their result can be read as follows.

\begin{theorem}{\rm (\cite{NLW})}\label{1}
Let $G$ be a connected graph of order $n$ with $n\geq5$, and $G$ contains no perfect
matchings. When $n$ is an even integer, the following statements hold.\\
$(i)$ If $n \leq 12$ and $\rho(G)\geq \rho(K_{\frac{n - 2}{2}} \vee \frac{n + 2}{2} K_{1})$, then $G$ is maximum matching covered unless $G=K_{\frac{n - 2}{2}} \vee \frac{n + 2}{2} K_{1}$.\\
 $(ii)$ If $ n \geq 14 $, if $\rho(G)\geq \rho( K_{2} \vee (K_{n - 5} \cup 3 K_{1}))$, then $G$ is maximum matching covered unless $G= K_{2} \vee (K_{n - 5} \cup 3 K_{1})$.\\
 When $n$ is an odd integer, the following statements hold.\\
$(i)$ If $n \leq 9$ and $\rho(G)\geq \rho(K_{\frac{n - 1}{2}} \vee \frac{n + 1}{2} K_{1})$, then $G$ is maximum matching covered unless $G=K_{\frac{n - 1}{2}} \vee \frac{n + 1}{2} K_{1}$.\\
 $(ii)$ If $ n \geq 11 $, if $\rho(G)\geq \rho( K_{2} \vee (K_{n - 4} \cup 2 K_{1}))$, then $G$ is maximum matching covered unless $G= K_{2} \vee (K_{n - 4} \cup 2 K_{1})$.
\end{theorem}

Motivated by Theorem \ref{1}, it is natural and interesting to ask the spectral radius conditions for graphs with bounded matching number to be maximum matching covered. The main result of this paper is the following Theorem \ref{main}. 

\begin{theorem}\label{main}
Assume that $k\geq1,n\geq k+4$ and $n\equiv k(\mod2)$. Let $G$ be a connected graph of order $n$ with matching number at most $\frac{n-k}{2}$. Then the following conclusions hold.\\
$(i)$ For $n \leq 3k + 6$, if $\rho(G)\geq \rho(K_{\frac{n - k}{2}} \vee \frac{n + k}{2} K_{1})$, then $G$ is maximum matching covered unless $G=K_{\frac{n - k}{2}} \vee \frac{n + k}{2} K_{1}$.\\
 $(ii)$ For $ n \geq 3k + 8 $, if $\rho(G)\geq \rho( K_{2} \vee (K_{n - k - 3} \cup (k + 1) K_{1}))$, then $G$ is maximum matching covered unless $G= K_{2} \vee (K_{n - k - 3} \cup (k + 1) K_{1})$.
\end{theorem}

Clearly, Theorem \ref{main} extends Theorem \ref{1} (by letting $k=2$ for even $n\geq6$, and letting $k=1$ for odd $n\geq5$).

The rest of the paper is organized as follows. In Section 2, we include some lemmas on spectral radius of graphs. In Section 3, we prove a useful lemma, which will be used in the proof of Theorem \ref{main}. In Section 4, we give the proof of Theorem \ref{main}.

\section{Spectral tools }

To prove the Theorem \ref{main}, we first include some lemmas. The first one in the following is taken from Theorem 2.2.1 of \cite{BH1}.

\begin{lem}(\cite{BH1})\label{subgraph}
Let $G$ be a connected graph, and let $H$ be a subgraph of $G$. Then $\rho(H)\leq\rho(G)$, and equality holds if and only if $H=G$.
\end{lem}

For a matrix (or a vector) $M$, let $M'$ denote the {\em transpose} of $M$. The following lemma is a variation of Theorem 8.1.3 of \cite{CRS}.

\begin{lem}(\cite{CRS})\label{eigenvector trans}
Let $G$ be a connected graph with a Perron vector $\mathbf{x}=(x_{1},x_{2},\ldots,x_{n})'$. Assume that $U$ and $W$ are two disjoint subsets of $V(G)$. Let $v\in V(G)-W$ such that there are no edges between $v$ and $W$. Let $G^{'}$ be the graph obtained from $G$ by deleting the edges between $v$ and $U-v$, and adding all the edges between $v$ and $W$. If $\sum_{u\in U}x_{u}\leq\sum_{w\in W}x_{w}$, then $\rho(G^{'})>\rho(G)$.
\end{lem}

For a graph $G$ and a partition $V_{1},V_{2},...,V_{m}$ of $V(G)$, the {\em quotient matrix} of this partition is the $m\times m$ matrix $(b_{ij})$, where $b_{ij}=\frac{e_{G}(V_{i},V_{j})}{|V_{i}|}$ for $i\neq j$, and $b_{ii}=\frac{2e(G[V_{i}])}{|V_{i}|}$ for $1\leq i\leq m$. The partition is {\em equitable} if each vertex in $V_{i}$ has the same number of neighbors in $V_{j}$ for any $1\leq i,j\leq m$. The following theorem is on eigenvalue interlacing technique (see \cite{BH1}, Chapter $3$).

\begin{theorem}(\cite{BH1})\label{interlacing}
Let $G$ be a graph on $n$ vertices and let $Q$ be the $m\times m$ quotient matrix of a partition $V_{1},V_{2},...,V_{m}$ of $V(G)$. If the partition is equitable, then $\rho(G)=\rho(Q)$, where $\rho(Q)$ is the largest eigenvalue of $Q$.
\end{theorem}

\section{A useful lemma}

Let $k\geq1$ and $n\geq k+4$ be two integers with the same parity (i.e., $n\equiv k(\mod2)$). Let $\mathcal{G}_{n,k}$ be the set of graphs $G$ of order $n$ with matching number at most $\frac{n-k}{2}$, such that there is a $S\subseteq V(G)$ with $|S|\geq2$ satisfying $c(G-S)\geq |S|+k$. Here, $c(G-S)$ denotes the number of components of $G-S$. Clearly, both $K_{\frac{n - k}{2}} \vee \frac{n + k}{2} K_{1}$  and $K_{2} \vee (K_{n - k - 3} \cup (k + 1) K_{1})$ are in $\mathcal{G}_{n,k}$.

\begin{lem}\label{main lemma}
 Let $\mathcal{G}_{n,k}$ be defined as above, where $k\geq1,n\geq k+4$ and $n\equiv k(\mod2)$. Then the following conclusions hold.\\
 $(i)$ For $n \leq 3k + 6$, $K_{\frac{n - k}{2}} \vee \frac{n + k}{2} K_{1}$ is the unique extremal graph with the maximum spectral radius in $\mathcal{G}_{n,k}$.\\
 $(ii)$ For $ n \geq 3k + 8 $,  $ K_{2} \vee (K_{n - k - 3} \cup (k + 1) K_{1})$ is the unique extremal graph with the maximum spectral radius in $\mathcal{G}_{n,k}$.
 \end{lem}
 
 \f{\bf Proof.} Let $G$ be an extremal graph with the maximum spectral radius in $\mathcal{G}_{n,k}$. We shall prove $G=K_{\frac{n - k}{2}} \vee \frac{n + k}{2} K_{1}$ for $n \leq 3k + 6$, and $G=K_{2} \vee (K_{n - k - 3} \cup (k + 1) K_{1})$ for $ n \geq 3k + 8 $. Since $G\in\mathcal{G}_{n,k}$, there is a $S\subseteq V(G)$ with $|S|\geq2$ such that $c(G-S)\geq |S|+k$. Set $s=|S|$. Let $Q_{1},Q_{2},...,Q_{q}$ be all the components of $G-S$. Then $G-S=\cup_{1\leq i\leq q}Q_{q}$, where $q\geq s+k$. Since $n\geq q+s\geq2s+k$, we have $s\leq\frac{n-k}{2}$. For $1\leq i\leq q$, set $n_{i}=|Q_{i}|$. Without loss of generality, assume that $n_{1}\geq n_{2}\geq\cdots\geq n_{q}\geq1$.
 
 \medskip
 
 \f{\bf Claim 1.} $G=K_{s}\vee(K_{n+1-2s-k}\cup (s+k-1)K_{1})$, where $2\leq s\leq\frac{n-k}{2}$.
 
 \medskip
 
\f{\bf Proof of Claim 1.} We first show that $G=K_{s}\vee(\cup_{1\leq i\leq q}K_{n_{i}})$. In fact, let $G_{1}$ be the graph obtained from $G$ by adding edges, so that $d_{G_{1}}(u)=n-1$ for any $u\in S$ and $Q_{i}$ is a complete graph for each $1\leq i\leq q$. Clearly, $G_{1}\in \mathcal{G}_{n,k}$. By Lemma \ref{subgraph}, we have $\rho(G)\leq \rho(G_{1})$ with equality only if $G=G_{1}$. By the choice of $G$, we must have $G=G_{1}$. 

Now we prove $q=s+k$. In fact, if $q>s+k$, let $G_{2}$ be the graph obtained from $G$ by connecting all the vertices of $Q_{q}$ to all the vertices of $Q_{1}$. Note that $c(G_{2}-S)=c(G-S)-1\geq s+k$. Thus,  $G_{2}\in\mathcal{G}_{n,k}$. However, by Lemma \ref{subgraph}, we have $\rho(G)<\rho(G_{2})$, a contradiction to the choice of $G$. Hence $q=s+k$.

It remains to show $n_{2}= n_{3}=\cdots= n_{q}=1$. Recall that $n_{1}\geq n_{2}\geq\cdots\geq n_{q}$. It suffices to prove $n_{2}=1$. By contradiction, suppose $n_{2}\geq2$. Let $\mathbf{x}=(x_{w})$ be a Perron vector of $G$, where $x_{w}$ is the element of $\mathbf{x}$ corresponding to vertex $w$.
Since $$\sum_{w\in V(Q_{i})}\rho(G)x_{w}=(n_{i}-1)\sum_{w\in V(Q_{i})}x_{w}+n_{i}\sum_{w\in S}x_{w}$$
for $1\leq i\leq q$, we have $$\sum_{w\in V(Q_{i})}x_{w}=\frac{n_{i}}{\rho(G)+1-n_{i}}\sum_{w\in S}x_{w}.$$
Hence $$\sum_{w\in V(Q_{1})}x_{w}\geq\sum_{w\in V(Q_{2})}x_{w}.$$
Choose a vertex $v\in V(Q_{2})$. Let $G_{3}$ be the graph obtained from $G$ by deleting the edges between $v$ and $V(Q_{2})-\left\{v\right\}$, and adding all edges between $v$ and $V(Q_{1})$. Note that $c(G_{3}-S)=c(G-S)= s+k$. Thus, $G_{3}\in \mathcal{G}_{n,k}$. By Lemma \ref{eigenvector trans}, we have $\rho(G)<\rho(G_{3})$, a contradiction to the choice of $G$. Hence $n_{2}=1$.
Consequently, we have $G=K_{s}\vee(K_{n+1-2s-k}\cup (s+k-1)K_{1})$, where $2\leq s\leq\frac{n-k}{2}$. This finishes the proof of Claim 1. \hfill$\Box$

\medskip

From Claim 1, we see $G=K_{s}\vee(K_{n+1-2s-k}\cup (s+k-1)K_{1})$, where $2\leq s\leq\frac{n-k}{2}$. Now define  
\[
H_{b} = K_{b} \vee \left( K_{n - k - 2b + 1} \cup (k + b - 1) K_{1} \right),
\]  where $2\leq b\leq \frac{n-k}{2}$. 
Let $$H_{0}=H_{2}=K_{2} \vee ( K_{n - k - 3} \cup (k + 1) K_{1})$$
 and 
 $$G_{0}=H_{\frac{n-k}{2}} = K_{\frac{n - k}{2}} \vee \frac{n + k}{2} K_{1}.$$
It suffices to prove $\rho(G_{0})>\rho(H_{b})$ for any $2\leq b< \frac{n-k}{2}$ when $n \leq 3k + 6$, and $\rho(H_{0})>\rho(H_{b})$ for any $2< b\leq \frac{n-k}{2}$ when $n \geq 3k + 8$. 

The quotient matrix \( M_{H_{b}} \) of the equitable partition $\left\{V(K_{b}),V(K_{n - k - 2b + 1}),V((k + b - 1) K_{1})\right\}$ of $V(H_{b})$ is given by  
\[
M_{H_{b}} = \begin{pmatrix}
b - 1 & n - k - 2b + 1 & k + b - 1 \\
b & n - k - 2b & 0 \\
b & 0 & 0
\end{pmatrix}.
\]
The characteristic polynomial \( \Phi_{H_{b}} \) of \( M_{H_{b}} \) is  
\[
\Phi_{H_{b}}(x) = x^{3} + (b - n + k + 1)x^{2} + \left( -b^{2} - bk + 2b - n + k \right)x + b(b + k - 1)(n - 2b - k).
\]
By Lemma \ref{interlacing}, $\rho(H_{b})$ is the largest root of $\Phi_{H_{b}}(x) $. 
We prove the lemma by the following two cases.

\medskip

\f{\bf Case 1.} $n \leq 3k + 6$.

\medskip

Recall that 
\[
G_{0} = K_{\frac{n - k}{2}} \vee \frac{n + k}{2} K_{1}.
\]  
In this case, we shall prove $\rho(G_{0})>\rho(H_{b})$ for any $2\leq b< \frac{n-k}{2}$. Note that $n\geq 2b+k+2$ as $b< \frac{n-k}{2}$.
For \( 2 \leq b < \frac{n - k}{2} \), recall that $\rho(H_{b})$ is the largest root of $$\Phi_{H_{b}}(x) = x^{3} + (b - n + k + 1)x^{2} + \left( -b^{2} - bk + 2b - n + k \right)x + b(b + k - 1)(n - 2b - k).$$  
Clearly, the quotient matrix \( M_{G_{0}} \) of the equitable partition $\left\{V(K_{\frac{n-k}{2}}),V(\frac{n+k}{2}K_{1})\right\}$ of $V(G_{0})$ is  
\[
M_{G_{0}} = \begin{pmatrix}
\frac{n - k - 2}{2} & \frac{n + k}{2} \\
\frac{n - k}{2} & 0
\end{pmatrix}.
\]
The characteristic polynomial \( \Phi_{G_{0}} \) of \( M_{G_{0}} \) is  
\[
\Phi_{G_{0}}(x) = x^{2} - \frac{n - k - 2}{2}x - \frac{n^{2} - k^{2}}{4}.
\]
By Lemma \ref{interlacing},  we have  
\[
\rho(G_{0}) = \frac{n - k - 2 + \sqrt{(n - k - 2)^{2} + 4(n^{2} - k^{2})}}{4}.
\]
We now prove  
\[
\frac{n - k - 2 + \sqrt{(n - k - 2)^{2} + 4(n^{2} - k^{2})}}{4} > n - k - 2.
\]  
In other words, we need to prove  
\[
n - k - 2 + \sqrt{(n - k - 2)^{2} + 4(n^{2} - k^{2})} > 4(n - k - 2).
\]
Equivalently,  
\[
\sqrt{(n - k - 2)^{2} + 4(n^{2} - k^{2})} > 3(n - k - 2),
\]  
\[
\Leftrightarrow(n - k - 2)^{2} + 4(n^{2} - k^{2}) > 9(n - k - 2)^{2},
\]   
\[
\Leftrightarrow n^{2} - k^{2} > 2(n - k - 2)^{2},
\]  
\[
\Leftrightarrow 0 > n^{2} + 3k^{2} - 4nk - 8n + 8k + 8
= n(n - 4k - 8) + 3k^{2} + 8k + 8.
\]

Recall that \( n \geq k+4 \).  
Let \( f(x) = x(x - 4k - 8) + 3k^2 + 8k + 8 \), where \( x \in [k+4, \, 3k+6] \). The function \( f(x) \) attains its maximum only at the endpoints of the interval. Recall \( k \geq 1 \). By a calculation, we have
$f(k + 4) = -8k - 8 < 0$ and
    $f(3k + 6) = -k - 4 < 0$.
    Thus \( f(x) < 0 \) for any $k+4\leq x\leq 3k+6$. It follows that $f(n)<0$, and then $\frac{n - k - 2 + \sqrt{(n - k - 2)^{2} + 4(n^{2} - k^{2})}}{4} > n - k - 2$ holds. Thus, we obtain
\[
\rho(G_{0}) > n - k - 2.
\]

By a calculation, we have  
\[
\Phi_{H_{b}}(x) = g(x) \Phi_{G_{0}}(x) + r(x),
\]  
where  
\[
g(x) = x - \frac{n - k - 2b}{2},
\]  
and  
\[
r(x)= \frac{n - k - 2b}{8}\left[ (4k + 4b - 4)x + 8b(b + k - 1) - n^{2} + k^{2} \right].
\]
Let  
\[
h(x) = (4k + 4b - 4)x + 8b(b + k - 1) - n^{2} + k^{2}.
\]  
Since \( k \geq 1 \) and \( b \geq 2 \), we have \( 4k + 4b - 4 > 0 \). Thus, \( h(x) \) is strictly increasing.
Then, for \( x > n - k - 2 \),  
\[
h(x) > h(n - k - 2) = n(4k + 4b - 4 - n) - 3k^{2} - 4k + 4bk - 16b + 8b^{2} + 8.
\]
Recall that \( n \in [2b + k + 2, 3k + 6] \) in this case. The minimum of \( h(n - k - 2) \) is attained only at the endpoints. Noting \( b \geq 2 \) and \( k \geq 1 \), we have
$$h(2b + k + 2 - k - 2)=12b^{2} + 12kb - 24b - 4k - 4\geq16>0,$$
   and
   $$h(3k+6 - k - 2)= k(16b - 28) + 8b + 8b^{2} - 52\geq0.$$
    Hence \( h(x) > h(n - k - 2) > 0 \) for $x > n - k - 2$.
Noting \( n - 2b - k > 0 \), we have  
\[
\Phi_{H_{b}}(x) - \left( x - \frac{n - k - 2b}{2} \right) \Phi_{G_{0}}(x) = \frac{n - k - 2b}{8} h(x) > 0, 
\] for any $x > n - k - 2$. Recall that $\rho(G_{0})>n-k-2$. Note that $ \Phi_{G_{0}}(x)\geq0$ for any $x\geq\rho(G_{0})$.
Thus,
$$\Phi_{H_{b}}(x)=\left( x - \frac{n - k - 2b}{2} \right) \Phi_{G_{0}}(x) + \frac{n - k - 2b}{8} h(x) >0$$ for any $x\geq\rho(G_{0})$.
This implies that  $\rho(H_{b})<\rho(G_{0})$, as desired.

\medskip

\f{\bf Case 2.} $n \geq 3k + 8$.

\medskip

Recall that 
\[
H_{0} = K_{2} \vee \left( H_{n - k - 3} \cup (k + 1) K_{1} \right).
\]  
In this case, we shall prove $\rho(H_{0})>\rho(H_{b})$ for any  $ 2 < b \leq \frac{n - k}{2}$.
For \( 3 \leq b \leq \frac{n - k}{2} \), recall that $\rho(H_{b})$ is the largest root of $$\Phi_{H_{b}}(x) = x^{3} + (b - n + k + 1)x^{2} + \left( -b^{2} - bk + 2b - n + k \right)x + b(b + k - 1)(n - 2b - k).$$  
Clearly, the quotient matrix \( M_{H_{0}} \) of  \( H_{0} \) becomes  
\[
M_{H_{0}} = \begin{pmatrix}
1 & n - k - 3 & k + 1 \\
2 & n - k - 4 & 0 \\
2 & 0 & 0
\end{pmatrix}.
\]
The characteristic polynomial \( \Phi_{H_{0}}(x) \) of \( M_{H_{0}} \) is  
\[
\Phi_{H_{0}}(x) = x^{3} - (n - k - 3)x^{2} - (n + k)x + 2(k + 1)(n - k - 4).
\]
$\rho(H_{0})$ is the largest root of $\Phi_{H_{0}}(x)$. Since \( K_{n - k - 1} \) is a proper subgraph of \( H_{0} \), by Lemma \ref{subgraph} we have  
\[
\rho(H_{0}) > \rho(K_{n - k - 1}) = n - k - 2.
\]

Let 
\[
p(x) = \Phi_{H_{b}}(x) - \Phi_{H_{0}}(x)
\]  
\[
= (b - 2)\left[ x^{2} - (b + k)x - 2b^{2} + (n - 3k - 2)b + n + kn - k^{2} - 5k - 4 \right].
\]
Let  
\[
q(x) = x^{2} - (b + k)x - 2b^{2} + (n - 3k - 2)b + n + kn - k^{2} - 5k - 4,
\]  
implying 
$$ p(x) = (b - 2) q(x).$$
Since $n\geq2b+k$ and $n\geq3k+8$, we have $2n\geq2b+4k+8$. It follows that $2(n-k-2)\geq2b+2k+4>2(b+k)$. Thus, $q(x)$ is strictly increasing for any $x > n - k - 2$. 
So, for \( x > n - k - 2 \), we have  
\[
q(x) > q(n - k - 2) = n^{2} - (2k + 3)n + k^{2} + k - 2bk - 2b^{2}.
\]
Recall that \( 3 \leq b \leq \frac{n - k}{2} \). For $x\in[3,\frac{n - k}{2} ]$, the polynomial $n^{2} - (2k + 3)n + k^{2} + k - 2kx - 2x^{2}$ attains its minimum only at the endpoints of  $[3,\frac{n - k}{2} ]$. Recall $n\geq3k+8$.
For $x = \frac{n - k}{2}$,
\[
    \begin{aligned}
    &n^{2} - (2k + 3)n + k^{2} + k - 2kx - 2x^{2} \\
    &= n^{2} - 2kn - 3n + k^{2} + k - (n - k)k - \frac{n^{2} - 2nk + k^{2}}{2} \\
    &= \frac{1}{2}\left[ (n - 4k - 6)n + 3k^{2} + 2k \right] \\
    &\geq \frac{1}{2}\left[ (3k + 8 - 4k - 6)(3k + 8) + 3k^{2} + 2k \right] \\
    &= \frac{1}{2}\left[ (-k + 2)(3k + 8) + 3k^{2} + 2k \right] \\
    &= \frac{1}{2}\left[ -3k^{2} - 8k + 6k + 16 + 3k^{2} + 2k \right] \\
    &= 8 > 0.
    \end{aligned}
    \]
      For $x=3$,
    \[
    \begin{aligned}
    &n^{2} - (2k + 3)n + k^{2} + k - 2kx - 2x^{2} \\
    &= n^{2} - (2k+3)n+ k^{2} - 5k - 18 \\
    &\geq (3k + 8)(k + 5) + k^{2} - 5k - 18 \\
    &= 3k^{2} + 15k + 8k + 40 + k^{2} - 5k - 18 \\
    &= 4k^{2} + 18k + 22 > 0.
    \end{aligned}
    \]
Thus, 
$$n^{2} - (2k + 3)n + k^{2} + k - 2kx - 2x^{2}>0$$
 for any $x\in[3,\frac{n - k}{2} ]$.
Then $$q(x) > q(n - k - 2) = n^{2} - (2k + 3)n + k^{2} + k - 2bk - 2b^{2}>0$$
for any $3\leq b\leq\frac{n - k}{2}$ and $x> n-k-2$.
Therefore,  
\[
p(x) = (b - 2) q(x) > (b - 2) q(n - k - 2) > 0
\]
for any $x> n-k-2$.
Recall \( \rho(H_{0}) > n - k - 2 \), we have  
\[
\Phi_{H_{b}}(x) - \Phi_{H_{0}}(x)=p(x)> 0,
\]  for any $x\geq\rho(H_{0})$.
Clearly, $\Phi_{H_{0}}(x)\geq0$ for any $x\geq\rho(H_{0})$.
Thus,
$\Phi_{H_{b}}(x) > 0$
 for any $x\geq\rho(H_{0})$. This implies that $\rho(H_{b})<\rho(H_{0})$, as desired.
This completes the proof. \hfill$\Box$

\section{Proof of Theorem \ref{main}}

To prove Theorem \ref{main}, we first introduce the Gallai–Edmonds Structure Theorem.
Let $G=[X, Y]$ be a bipartite graph. For any $S\subseteq X$, let \( |N_{G}(S)| - |S| \) be called the \textit{surplus} of \( S \). The minimum surplus over all nonempty subsets of \( X\) is called the \textit{\(X\)-surplus} of the graph \( G \). 

 Let \( G \) be a graph. Let
$D(G)$ be the set of vertices $w$, such that there is a maximum matching which does not cover $w$.
Let \( B(G) \) denote the set of vertices $v$ in \( V(G) \setminus D(G) \), such that $v$ is adjacent to at least one vertex in \( D(G) \). Let \( C(G) = V(G) \setminus (D(G) \cup B(G)) \).

\begin{lem} {\rm (Gallai–Edmonds Structure Theorem \cite{LP})} For a graph \( G \) , let  \( D(G), B(G)\) and \(C(G) \) be defined as above. Then the following statements hold:
\begin{itemize}
    \item[\emph{(i)}] Each component of the subgraph induced by \( D(G) \) is factor-critical.
    \item[\emph{(ii)}] The subgraph induced by \( C(G) \) has a perfect matching.
    \item[\emph{(iii)}] After deleting the vertices of \( C(G) \) and the edges inside \( B(G) \) from \( G \), and contracting each component of \( D(G) \) into a single vertex, the resulting bipartite graph has positive surplus (as viewed from \( B(G) \)).
    \item[\emph{(iv)}] Every maximum matching of \( G \) contains a near-perfect matching of each component of \( D(G) \) and a perfect matching of each component of \( C(G) \), and each vertex in \( B(G) \) is paired with a vertex from a distinct component of \( D(G) \).
\end{itemize}
\end{lem}

 $\emptyset$ is stipulated as an independent set. Using Gallai–Edmonds Structure Theorem, Niu, Lian and Wang \cite{NLW} obtained the following result. 

\begin{lem}{\rm (\cite{NLW})}\label{C(G)=0}  
Let \( G \) be a connected graph without a perfect matching. Then \( G \) is a maximum matching covered graph if and only if \( C(G) = \varnothing \) and \( B(G) \) is an independent set.
\end{lem}

Now we are ready to prove Theorem \ref{main}

\medskip

\f{\bf Proof of Theorem \ref{main}.} Recall that $k\geq1,n\geq k+4$ and $n\equiv k(\mod2)$. Observe that both $K_{\frac{n - k}{2}} \vee \frac{n + k}{2} K_{1}$ and $K_{2} \vee (K_{n - k - 3} \cup (k + 1) K_{1})$  have matching number $\frac{n-k}{2}$ (see Berge-Tutte Formula \cite{W}). Clearly, they both are not maximum matching covered, since an edge inside $V(K_{\frac{n-k}{2}})$ or $V(K_{2})$ is not contained in a maximum matching. 
  Recall that $\rho(G)\geq\rho(K_{\frac{n - k}{2}} \vee \frac{n + k}{2} K_{1})$ for $n \leq 3k + 6$, and $\rho(G)\geq\rho(K_{2} \vee (K_{n - k - 3} \cup (k + 1) K_{1}))$ for $ n \geq 3k + 8 $.
Suppose that  $G$  is not maximum matching covered. It suffices to prove $G=K_{\frac{n - k}{2}} \vee \frac{n + k}{2} K_{1}$ for $n \leq 3k + 6$, and $G=K_{2} \vee (K_{n - k - 3} \cup (k + 1) K_{1})$ for $ n \geq 3k + 8 $. Recall that $G$ has matching number at most $\frac{n-k}{2}$. Thus $G$ contains no perfect matching.

Let $D(G), B(G)$ and $C(G)$ be the partition of $V(G)$ in the Gallai–Edmonds Structure Theorem. Clearly, there are no edges between $D(G)$ and $C(G)$. Let $Q_{1},Q_{2},...,Q_{q}$ be the components of the subgraph induced by $D(G)$. Note that $|Q_{i}|$ is odd for any $1\leq i\leq q$, since $Q_{i}$ is factor-critical. Set $b=|B(G)|$. 
From the construction of maximum matchings in Gallai–Edmonds Structure Theorem, we see that $G$ has matching number $\frac{n-(q-b)}{2}$. Thus, $\frac{n-(q-b)}{2}\leq\frac{n-k}{2}$, implying 
$$q\geq b+k.$$
 Since $G$ is connected and not maximum matching covered, by Lemma \ref{C(G)=0}, we have $C(G)\neq\emptyset$ or $B(G)$ is not independent. Thus, we have the following two cases.

\medskip

\f{\bf Case 1.}  $C(G)\neq\emptyset$.

\medskip

Note that $|C(G)|\geq2$ in this case, since $|C(G)|$ is even. Since $q\geq k+b\geq1$ and $G$ is connected, we must have $b\geq1$ as $C(G)\neq\emptyset$. Now choose a vertex $u\in C(G)$. Let $S=B(G)\cup\left\{u\right\}$. Then $|S|\geq2$. Moreover, noting $C(G)-\left\{u\right\}\neq\emptyset$, we have $c(G-S)\geq q+1\geq b+k+1=|S|+k$. It follows that $G\in\mathcal{G}_{n,k}$ (defined in Lemma \ref{main lemma}). By Lemma \ref{main lemma},
$K_{\frac{n - k}{2}} \vee \frac{n + k}{2} K_{1}$ is the unique graph with the maximum spectral radius in $\mathcal{G}_{n,k}$ for $n \leq 3k + 6$, and $K_{2} \vee (K_{n - k - 3} \cup (k + 1) K_{1})$  is the unique graph with the maximum spectral radius in $\mathcal{G}_{n,k}$ for $n \geq 3k + 8 $. Thus, we must have $G=K_{\frac{n - k}{2}} \vee \frac{n + k}{2} K_{1}$ for $n \leq 3k + 6$, and $G=K_{2} \vee (K_{n - k - 3} \cup (k + 1) K_{1})$ for $ n \geq 3k + 8 $, as desired.

\medskip

\f{\bf Case 2.}  $B(G)$ is not independent.

\medskip

Recall that $\emptyset$ is stipulated as an independent set. Thus $b\geq2$ as $B(G)$ is not independent. Recall that $c(G-B(G))\geq q\geq b+k$. Thus, $G\in\mathcal{G}_{n,k}$ (defined in Lemma \ref{main lemma}). Similar to Case 1, we have $G=K_{\frac{n - k}{2}} \vee \frac{n + k}{2} K_{1}$ for $n \leq 3k + 6$, and $G=K_{2} \vee (K_{n - k - 3} \cup (k + 1) K_{1})$ for $ n \geq 3k + 8 $, as desired. This completes the proof. 
\hfill$\Box$

\medskip

\f{\bf Data availability statement}

\medskip

There is no associated data.

\medskip

\f{\bf Declaration of Interest Statement}

\medskip

There is no conflict of interest.

\end{document}